\documentclass[11pt]{article}
\usepackage[margin=1.5in]{geometry}
\usepackage{amsthm,amsmath,amssymb}
\usepackage{graphicx}
\usepackage{color}
\usepackage{enumerate}
\usepackage{booktabs}
\usepackage{hyperref}

\theoremstyle{definition}
\newtheorem{theorem}{Theorem}
\newtheorem{lemma}[theorem]{Lemma}

\newtheorem{proposition}[theorem]{Proposition}

\newtheorem{definition}[theorem]{Definition}

\newtheorem{conjecture}[theorem]{Conjecture}

\theoremstyle{remark}
\newtheorem{remark}[theorem]{Remark}

\newcommand{\BN}{\mathbb{N}}
\newcommand{\Mod}[1]{\ (\mathrm{mod}\ #1)}
\begingroup
\DeclareMathOperator{\AP}{AP}

\title{Functions on Antipower Prefix Lengths of the Thue-Morse Word}
\author{Shyam Narayanan\\
\small Department of Mathematics\\[-0.8ex]
\small Harvard University\\[-0.8ex] 
\small Cambridge, MA, U.S.A.\\
\small\tt shyam.s.narayanan@gmail.com
}

\date{}

\begin{document}

\maketitle

\begin{abstract}
	We say that a word $w$ of length $kn$ is a $k$-\textit{antipower} if it can be written in the form $w_1 \cdots w_k$, where each $w_i$ is a distinct word of length $n$. We analyze prefixes of the Thue-Morse word $\textbf{t}$ and lengths of antipowers occurring in them. Define $\Gamma(k)$ to be the largest odd $n$ such that the prefix of $\textbf{t}$ of length $kn$ is not a $k$-antipower, and $\gamma(k)$ to be the smallest odd $n$ such that the corresponding prefix is a $k$-antipower. We provide strong bounds on the asymptotic values of $\gamma(k)$ and $\Gamma(k)-\gamma(k)$. Our bounds on $\gamma(k)$ affirmatively answer one conjecture of Defant and make substantial progress towards answering a second conjecture of Defant. It was previously known that $\Gamma(k)$ and $\gamma(k)$ grow linearly in $k$, but our bounds on $\Gamma(k)-\gamma(k)$ prove that $\Gamma(k)-\gamma(k)$ also grows linearly in $k$.
\end{abstract}

\section{Introduction}

A finite word $W$, i.e., a finite string of letters from a fixed alphabet, is called a \emph{$k$-power} if $W = w^k = ww\cdots w$ (concatenated $k$ times), where $w$ is another word. Thue \cite{Thue} created an infinite binary word (i.e., a word with only two distinct letters) such that no finite contiguous subword of the word is a $3$-power. This word, now famously known as the Thue-Morse word, can be defined as follows:

\begin{definition}
    For each $\ell\in\{0,1\}$, let $\overline{\ell}=1-\ell$. If $W=\ell_1\cdots\ell_k$ is a finite word of length $k$ over the alphabet $\{0,1\}$, let $\overline{W}=\overline{\ell_1}\cdots\overline{\ell_k}$. Consider the sequence $A_n$ of finite words over the alphabet $\{0, 1\}$ such that $A_0 = 0$ and $A_n = A_{n-1}\overline{A_{n-1}}$ for all $n\geq 1$. Define the \emph{Thue-Morse word} $\textbf{t}$ to be $$\textbf{t} = \lim\limits_{n \to \infty} A_n = 0110100110010110\cdots.$$ Let $\textbf{t} = t_0t_1t_2\dots$, where $t_n$ is the $(n+1)^\text{st}$ letter in $\textbf{t}$.
\end{definition}

The Thue-Morse word is known to be very useful in a variety of fields, such as combinatorics, economics \cite{PALACIOS-HUERTA}, game theory \cite{Cooper}, and analytic number theory \cite{AlloucheCohen}. Various sequences and constants relating to the Thue-Morse sequence have also been studied in detail \cite{Dekking,Mahler}.

The Thue-Morse word is known to be \textit{overlap-free}, meaning that for all finite words $x, y$ such that $x$ is nonempty, the word $xyxyx$ is never a contiguous subword of $\textbf{t}$ \cite{Thue}. This condition implies that $\textbf{t}$ does not contain any $3$-powers as contiguous subwords, since letting $y$ be the empty word tells us that $xxx$ can never be a contiguous subword of $\textbf{t}.$ Because the Thue-Morse word contains no $3$-powers as contiguous subwords, looking at $k$-powers in the Thue-Morse word is not as interesting as looking at what are called $k$-antipowers, first introduced by Fici, Restivo, Silva, and Zamboni \cite{OrigPaper}.

\begin{definition}
    A word $W$ of length $kn$ is called a \emph{$k$-antipower} if $W = w_1\cdots w_k$, where $|w_i| = n$ for all $1 \le i \le k$ and $w_i \neq w_j$ for all $1 \le i \neq j \le k$.
\end{definition}

Suppose $x$ is a word and $k$ is a positive integer. We consider the set $\AP(\emph{x}, k) \subset \mathbb{N}$, defined to be the set of positive integers $n$ such that the prefix of length $kn$ in $x$ (i.e., the word formed from the first $kn$ letters of $\emph{x}$) is a $k$-antipower. Note that replacing each occurrence of $0$ with $01$ and each occurrence of $1$ with $10$ in $\textbf{t}$ again gives us $\textbf{t},$ so for all $i, j,$ we have $t_{(i-1)n} \cdots t_{in-1} = t_{(j-1)n} \cdots t_{jn-1}$ if and only if $t_{2(i-1)n} \cdots t_{2in-1} = t_{2(j-1)n} \cdots t_{2jn-1}$. This means that $n \in \AP(\textbf{t}, k)$ if and only if $2n \in \AP(\textbf{t}, k)$. Consequently, we are mainly interested in $\AP(\textbf{t}, k) \cap (2\mathbb{N}-1)$, the set of odd $n\in \AP({\bf t},k)$.

\begin{definition}
    For $k \ge 1,$ let $\mathcal{F}(k) = \AP(\textbf{t}, k) \cap (2\mathbb{N}-1)$ be the set of odd $n$ such that the prefix of $\textbf{t}$ of length $kn$ is a $k$-antipower.
\end{definition}

It turns out that $\AP(\textbf{t}, k)$, and thus $\mathcal{F}(k)$, is nonempty for all $k$ \cite{OrigPaper}. In fact, we have $(2\mathbb{N}-1) \backslash \mathcal{F}(k)$ is finite for all $k$ \cite{Defant}. As a result, we can define the following:

\begin{definition}
    Define $\gamma(k)$ to be the minimum element in $\mathcal{F}(k)$ and $\Gamma(k)$ to be the maximum element in $(2\mathbb{N}-1) \backslash \mathcal{F}(k)$. To avoid issues of the maximum of the empty set, we define $\Gamma(k) = 0$ if $\AP(\textbf{t}, k) = 2\BN - 1.$
\end{definition}

\begin{remark} \label{3inF}
    For $k \ge 3,$ note that $\Gamma(k) \neq 0.$ This is true since $\textbf{t} = 011010011 \cdots $ and since $t_0t_1t_2 = t_6t_7t_8,$ the prefix of length $3k$ is not a $k$-antipower for all $k \ge 3.$ Therefore, we have $3 \not\in \AP(\textbf{t}, k)$ for all $k \ge 3$, so $\Gamma(k) \ge 3$ for all $k \ge 3$.
\end{remark}

To help us understand $\Gamma$ and $\gamma$, we will analyze the following natural function:

\begin{definition}
    For an odd positive integer $n \ge 3$, define $\mathfrak{K}(n)$ to be the smallest $k$ such that the prefix of $\textbf{t}$ of length $kn$ is not a $k$-antipower. Our definition does not make sense for $n = 1$ since every prefix of $\textbf{t}$ is a $1$-antipower, so define $\mathfrak{K}(1) = 0.$
\end{definition}

Note that for $k \ge 3,$ we have that $\Gamma(k)$ is the largest odd $n$ such that the prefix of length $kn$ is not a $k$-antipower, which means $\Gamma(k)$ is the largest odd $n$ such that $\mathfrak{K}(n) \le k.$ Likewise, for $k \ge 3,$ we have that $\gamma(k)$ is the smallest odd $n$ such that the prefix of length $kn$ is a $k$-antipower. Since $\gamma(k)$ does not equal $1$ for $k \ge 3$ by Remark \ref{3inF}, this means $\gamma(k)$ is the smallest odd $n$ such that $\mathfrak{K}(n) > k.$

\begin{definition}
    For positive integers $c$ and $n$, define the $c^\text{th}$ \textit{block} of size $n$ in $\textbf{t}$ to be the subword $t_{(c-1)n}t_{(c-1)n+1}\cdots t_{cn-1}.$
\end{definition}

\begin{remark}
    Note that if $\mathfrak{K}(n) = k,$ then the first $k-1$ blocks of size $n$ are distinct, but the $k^\text{th}$ block of size $n$ equals one of the first $k-1$ blocks of size $n$.
\end{remark}

A basic application of the pigeonhole principle shows that $\mathfrak{K}(n) \le 2^n+1$. However, this bound is extremely poor. In fact, it is known that $\mathfrak{K}(n)$ grows linearly in $n$ for odd $n$, meaning $\limsup_{n \to \infty} \frac{\mathfrak{K}(n)}{n} < \infty$ and $\liminf_{n \to \infty} \frac{\mathfrak{K}(n)}{n} > 0.$ It is also known that $\gamma(k)$ and $\Gamma(k)$ grow linearly in $k$, and Defant \cite{Defant} proved the following, up to a small error in parts (a) and (b), which we explain and fix in \ref{FixColin}:

\begin{theorem} \cite{Defant} \label{DefantTheorem} \hspace{.5cm}
\begin{enumerate}[(a)]
    \item $\frac{1}{2} \le \liminf\limits_{k \to \infty} \frac{\gamma(k)}{k} \le \frac{9}{10}$,
    \item $1 \le \limsup\limits_{k \to \infty} \frac{\gamma(k)}{k} \le \frac{3}{2}$,
    \item $\liminf\limits_{k \to \infty} \frac{\Gamma(k)}{k} = \frac{3}{2}$,
    \item $\limsup\limits_{k \to \infty} \frac{\Gamma(k)}{k} = 3$.
\end{enumerate}
\end{theorem}

The growth of $\gamma(k)$ is not as well understood as that of $\Gamma(k).$ Defant made the following conjecture about the growth of $\gamma(k)$:

\begin{conjecture} \label{DefantConjecture} \cite{Defant}\hspace{.5cm}
\begin{enumerate}[(a)]
    \item $\liminf\limits_{k \to \infty} \frac{\gamma(k)}{k} = \frac{9}{10}$,
    \item $\limsup\limits_{k \to \infty} \frac{\gamma(k)}{k} = \frac{3}{2}$.
\end{enumerate}
\end{conjecture}

In Section \ref{Preliminary} of our paper, we note some simpler propositions which end up being very useful for Section \ref{BoundsK}, where we provide bounds for $\mathfrak{K}(n)$ for odd integers $n$. In Section \ref{BoundsGamma}, we use our results from Section \ref{BoundsK} to prove Conjecture \ref{DefantConjecture} (b) and improve the lower bound for $\liminf\limits_{k \to \infty} \frac{\gamma(k)}{k}$ to $\frac{3}{4}$, thus making progress towards solving Conjecture \ref{DefantConjecture} (a). As suggested in \cite{Defant}, we also study $\Gamma(k)-\gamma(k)$ and show that $\Gamma(k)-\gamma(k)$ grows linearly in $k$ in Section \ref{BoundsGamma}. We summarize our results in the following theorem:

\begin{theorem} \hspace{.5cm}
\begin{enumerate}[(a)]
    \item $\frac{3}{4} \le \liminf\limits_{k \to \infty} \frac{\gamma(k)}{k} \le \frac{9}{10}$,
    \item $\limsup\limits_{k \to \infty} \frac{\gamma(k)}{k} = \frac{3}{2}$,
    \item $\frac{1}{2} \le \liminf\limits_{k \to \infty} \frac{\Gamma(k)-\gamma(k)}{k} \le \frac{3}{4}$,
    \item $\frac{11}{6} \le \limsup\limits_{k \to \infty} \frac{\Gamma(k)-\gamma(k)}{k} \le \frac{9}{4}$.
\end{enumerate}
\end{theorem}

Understanding the asymptotic values of $\Gamma(k)-\gamma(k)$ is quite interesting, because if $\Gamma(k)-\gamma(k) = o(k)$ for some value of $k$, then $\mathcal{F}(k)$ contains all odd positive integers less than some $n_1$ but no odd integer greater than some $n_2,$ where $n_2 = (1+o(1)) n_1.$ Thus, all elements of $\mathcal{F}(k)$ are tightly clustered together. However, our proof that $\Gamma(k)-\gamma(k)$ grows linearly in $k$ implies that such a phenomenon never happens.

For visualizations of the growth of $\mathfrak{K}(n)$ as well as $\gamma(k)$ and $\Gamma(k),$ we point the reader to the end of \cite{Defant}, which provides plots describing the growth of $\gamma(k)/k, \Gamma(k)/k,$ and $\mathfrak{K}(n).$ We also provide a visualization of $\mathfrak{K}(n)$ in Section \ref{Conclusion} to provide intuition of the behavior of $\mathfrak{K}(n)$ for large $n$ and to support a conjecture we make.

\section{Preliminary Results} \label{Preliminary}

In this section, we establish some important propositions that are useful in proving lemmas for bounding $\mathfrak{K}(n)$ in the next section. We first note the following well-known result (see for example \cite[Proposition 1]{AlloucheShallit1999}):

\begin{proposition} \label{Basic}
    The sum of the digits of $n$ in base $2$, reduced modulo $2$, equals $t_n$.
\end{proposition}

Using Proposition \ref{Basic}, we make the following definition and observation.

\begin{definition}
    For positive integers $m, n,$ we say that $n$ is \emph{equivalent} to $m$ if $t_n = t_m$. We denote this by $n \equiv_t m$. Note that $n \equiv_t m$ if and only if the binary digit sums of $n$ and $m$ are congruent $\bmod$ $2$.
\end{definition}

The following result gives us a cleaner way to compare the $c^\text{th}$ and $(c')^\text{th}$ blocks of size $n$, given $c$ and $c'$ differ by a power of $2$. We use the following proposition several times to show that the $c^\text{th}$ and $(c+2^i)^\text{th}$ blocks are the same for some $i$ and some sufficiently small $c$, which will give us upper bounds for $\mathfrak{K}(n)$.

\begin{proposition} \label{DivideBy2^i}
    Let $n$ be odd and $i$ be a nonnegative integer. For all nonnegative integers $c$, the $(c+1)^\text{th}$ and $(c+1+2^i)^\text{th}$ blocks of size $n$ in ${\bf t}$ are equal if and only if $x \equiv_t x+n$ for every integer $x$ such that $\left\lfloor\frac{cn}{2^i}\right\rfloor \le x \le \left\lfloor\frac{(c+1)n-1}{2^i}\right\rfloor$.
\end{proposition}

\begin{proof}
    For the ``if'' direction, it suffices to prove that $y \equiv_t y+2^in$ for all $y$ such that $cn \le y \le cn+n-1$. Note that $y \equiv y+2^in \Mod {2^i}$. By the binary digit sum definition of ${\bf t}$, it suffices to show that $\left\lfloor\frac{y}{2^i}\right\rfloor \equiv_t \left\lfloor\frac{y+2^in}{2^i}\right\rfloor = \left\lfloor\frac{y}{2^i}\right\rfloor+n.$ But we know that $\left\lfloor\frac{y}{2^i}\right\rfloor$ is between $\left\lfloor\frac{cn}{2^i}\right\rfloor$ and $\left\lfloor\frac{(c+1)n-1}{2^i}\right\rfloor$, so we are done.
    
    
    The ``only if" direction follows from the fact that if the $(c+1)^\text{th}$ and the $(c+1+2^i)^\text{th}$ blocks match, then $cn+(a-1) \equiv_t (c+2^i)n+(a-1)$ for all $1 \le a \le n,$ because the $a^\text{th}$ element of the $(c+1)^\text{th}$ and $(c+1+2^i)^\text{th}$ length-$n$ blocks of positive integers must match. However, as $cn+(a-1) \equiv (c+2^i)n+(a-1) {\Mod {2^i}},$ it follows from the binary digit sum definition of $t_i$ that $\left\lfloor\frac{cn+(a-1)}{2^i}\right\rfloor \equiv_t \left\lfloor\frac{(c+2^i)n+(a-1)}{2^i}\right\rfloor.$ This can be thought of as chopping off the last $i$ digits and checking Thue-Morse equivalence. However, we clearly have $\left\lfloor\frac{(c+2^i)n+(a-1)}{2^i}\right\rfloor = \left\lfloor\frac{cn+(a-1)}{2^i}\right\rfloor + n,$ so letting $a$ vary from $1$ to $n$ gives the desired result.
\end{proof}

The following proposition implies that $\mathfrak{K}(n) \ge 1+2^{1+\left\lfloor \log_2(n/3)\right\rfloor}$ for $n \ge 5$, and is a direct consequence of \cite[Proposition 6(a)]{Defant}.

\begin{proposition} \label{ConvOfPrev}
    Let $n$ be odd and $i$ be a positive integer. Suppose that $c, c'$ are positive integers such that the $c^\text{th}$ and $(c')^\text{th}$ blocks of size $n$ in the Thue-Morse word are equal. If $3 \cdot 2^{i-1} < n,$ then $c \equiv c' \Mod {2^i}.$
\end{proposition}

Finally, we note the following:

\begin{proposition} \label{NotEquivalent}
    For all odd positive integers $m$, there exists $y \in \{1, 2\}$ such that $y \equiv_t y+m$ and $y' \in \{1, 2\}$ such that $y' \not\equiv_t y'+m.$ Thus, if $2^j|m$ but $2^{j+1} \nmid m$ for some integer $j$, there exists a positive integer $y \le 2^{j+1}$ such that $y \equiv_t y+m$ and a positive integer $y' \le 2^{j+1}$ such that $y' \not\equiv_t y'+m.$
\end{proposition}

\begin{proof}
    The first statement follows by the fact that $1 \equiv_t 2$ but $m+1 \not\equiv_t m+2$ since $m+1$ is even. Therefore, $y$ and $y'$ will equal $1$ and $2$ in some order. The second statement is immediate by looking at $y, y'$ equaling $2^j$ and $2^{j+1}$ in some order, after we note that $y \equiv_t y+m$ if and only if $\frac{y}{2^j} \equiv_t \frac{y+m}{2^j} = \frac{y}{2^j} + \frac{m}{2^j}$. Here, we use the fact that $\frac{y}{2^j}$ is an integer and $\frac{m}{2^j}$ is an odd integer.
\end{proof}

\section{Bounds on $\mathfrak{K}(n)$} \label{BoundsK}

To establish bounds for $\gamma(k)$ and $\Gamma(k)-\gamma(k)$, we prove many lemmas bounding $\mathfrak{K}(n)$. We remark that $n$ is always odd in what follows.

Lemmas \ref{8x+1}, \ref{32x+d}, \ref{CloseToHighV2}, and \ref{2^i+d} allow us to bound $\mathfrak{K}(n)$ for all odd $n$ between $2^i$ and $3 \cdot 2^{i-1} - 13$, which will in turn allow us to bound $\gamma(k).$ Lemmas \ref{8x+1}, \ref{32x+d}, \ref{CloseToHighV2}, and \ref{2^i+d} bound $\mathfrak{K}(n)$ by looking at various cases for $n$ based on the residue of $n$ modulo $32$ as well as whether $n$ is very close to a multiple of a large power of $2$.

\begin{lemma} \label{8x+1}
    Suppose that $2^i < n < 3 \cdot 2^{i-1}$ and $i$ is sufficiently large. If $n = a \cdot 2^j+1$ or $n=a\cdot 2^j-1$, where $j \ge 3$ and $a$ is odd, then $\mathfrak{K}(n) \le 2^i + 2^{j+1} + 5$.
\end{lemma}

\begin{proof}
    Define $m = a \cdot 2^{j-3}$. If $n = a \cdot 2^j+1$, then $n = 8m+1$, and if $n = a \cdot 2^j-1$, then $n = 8m-1$.
    
    Since $2^{j-3}$ is the largest power of $2$ dividing $m$, there exists $1 \le y \le 2^{j-2}$ such that $y \not\equiv_t y+m$ by Proposition \ref{NotEquivalent}. Now, let $y$ be the smallest nonnegative integer such that $y \not\equiv_t y+m$, and let $x = 8y+3$ if $n = 8m-1$ and $x = 8y+2$ if $n = 8m+1$.
    
    If $n = 8m-1$, then
    \begin{align*}
        8y+3 & \equiv_t (8y+3)+(8m-1),\\
        8y+4 & \equiv_t (8y+4)+(8m-1),\\
        8y+5 & \equiv_t (8y+5)+(8m-1).
    \end{align*}  
    Similarly, if $n = 8m+1,$ then
    \begin{align*}
        8y+2 & \equiv_t (8y+2)+(8m+1),\\
        8y+3 & \equiv_t (8y+3)+(8m+1),\\
        8y+4 & \equiv_t (8y+4)+(8m+1).
    \end{align*}
    
    In either case, this shows that $x+s\equiv_tx+n+s$ for $s=0,1,2$. Let $c$ be the smallest integer such that $\left\lfloor \frac{cn}{2^i}\right\rfloor \ge x.$ If $\left\lfloor\frac{cn}{2^i}\right\rfloor = x$, then $\left\lfloor\frac{(c+1)n}{2^i}\right\rfloor \le x+2$, so it follows from Proposition \ref{DivideBy2^i} that the $(c+1)^\text{th}$ and $(c+1+2^i)^\text{th}$ blocks of ${\bf t}$ match. Else, $\left\lfloor\frac{cn}{2^i}\right\rfloor = x+1,$ so either $\left\lfloor\frac{(c+1)n}{2^i}\right\rfloor = x+2$ or $\left\lfloor\frac{(c-1)n}{2^i}\right\rfloor = x$ since $\frac{2n}{2^i} < 3$. Thus, either the $(c+1)^\text{th}$ block matches with the $(c+1+2^i)^\text{th}$ block or the $c^\text{th}$ block matches $(c+2^i)^\text{th}$ block. Because $c+1 \le x+2$, we can take $c' \in \{c, c+1\}$ such that $c' \le x+2 \le 8y+5$ and the $(c')^\text{th}$ and $(c'+2^i)^\text{th}$ blocks match.
    
    Since $y \le 2^{j-2}$ and $x \le 2^{j+1}+3,$ there is some $c' \le 2^{j+1} + 5$ such that the $(c')^\text{th}$ and $(c'+2^i)^\text{th}$ blocks match. This means that $\mathfrak{K}(n) \le 2^i + 2^{j+1} + 5$.
\end{proof}

\begin{lemma} \label{32x+d}
    Suppose that $2^i < n < 3 \cdot 2^{i-1}$ and $i$ is sufficiently large. Also, suppose that $n = a \cdot 2^j+d,$ where $j \ge 5$, $a$ is odd, and $d \in \{\pm 3, \pm 5, \pm 11, \pm 13\}$. Then $\mathfrak{K}(n) \le 2^i + 2^{j+1} + 28$.
\end{lemma}

\begin{proof}
    Define $m = a \cdot 2^{j-5},$ which is an integer since $j \ge 5$. Now define the function $\varphi$ such that $$\varphi(d) = 
    \begin{cases} 
    15, & \text{if } d \in \{3, 11, -5, -13\}; \\
    10, & \text{if } d = 5; \\
    2, & \text{if } d \in \{1, 13\}; \\
    3, & \text{if } d = -1; \\
    18, & \text{if } d = -3; \\
    26, & \text{if } d = -11. \\ 
   \end{cases} $$
   
   The values of $\varphi$ at $1$ and $-1$ are not used for this lemma, but for Lemma \ref{CloseToHighV2}.
   
   It is straightforward to verify that for all $d \in \{3, -3, 5, -5\}$ and $s \in \{0, 1, 2\}$, we have that $d+\varphi(d)+s \equiv_t \varphi(d)+s$. In addition, one can verify that for all $d \in \{1, -1, 11, 13, -11, -13\}$ and $s \in \{0, 1, 2\}$, we have $d+\varphi(d)+s \not\equiv_t \varphi(d)+s$.
   
   Also, since $0 \le \varphi(d), d+\varphi(d) < 30,$ the binary sum definition of ${\bf t}$ tells us that if $x, y$ are positive integers such that $x = 32y + \varphi(d)$, then $$t_{x+s}=t_{32y+\varphi(d)+s} \equiv t_y+t_{\varphi(d)+s} \Mod 2$$ and $$t_{x+n+s} = t_{32y+a\cdot 2^{j}+(d+\varphi(d)+s)} \equiv t_{y+m}+t_{d+\varphi(d)+s} \Mod 2$$ for $s \in \{0, 1, 2\}.$ Therefore, if $y \equiv_t y+m$ and if $d \in \{3, -3, 5, -5\},$ then we must have $x+s \equiv_t x+n+s$ for $x = 32y+\varphi(d)$ and for all $s \in \{0, 1, 2\}$. Similarly, if $y \not\equiv_t y+m$ and if $d \in \{11, 13, -11, -13\},$ then $x+s \equiv_t x+n+s$ for all $s \in \{0,1,2\}$.
   
   As $2^{j-5}$ is the largest power of $2$ dividing $m$, there exist positive integers $y_1, y_2 \le 2^{j-4}$ such that $y_1 \not\equiv_t y_1+m$ and $y_2 \equiv_t y_2+m$. Let $y = y_1$ if $d \in \{11, 13, -11, -13\}$ and $y = y_2$ if $d \in \{3, 5, -3, -5\},$ and let $x = 32 y + \varphi(d).$ As in the proof of the previous lemma, there is some integer $c \le x+2$ such that the $c^\text{th}$ and $(c+2^i)^\text{th}$ blocks of ${\bf t}$ match, since there exists an integer $c \le x+2$ such that $\lfloor \frac{cn}{2^i} \rfloor \le x+2$ but $\lfloor \frac{(c-1)n}{2^i} \rfloor \ge x.$ Consequently, we have that $\mathfrak{K}(n) \le 2^i + 2^{j+1} + \max(\varphi(d))+2=2^i + 2^{j+1} + 28.$
\end{proof}

In the following lemma, we will assume that $2^{i-1} < n < 3 \cdot 2^{i-1}$ rather than $2^i < n < 3 \cdot 2^{i-1}.$ We allow a weaker assumption to fix an error in \cite{Defant}, which will be explained in detail in \ref{FixColin}.

\begin{lemma} \label{CloseToHighV2}
    Let $n = a \cdot 2^j+d$, where $a$ is odd and $d \in \{\pm 1, \pm 3, \pm 5, \pm 11, \pm 13\}$. Suppose that $2^{i-1} < n < 3 \cdot 2^{i-1}$ for some sufficiently large $i$, and define $k = i-j.$ If $j \ge \frac{3i+11}{4}$, and if $k \ge 2$ when $2^i < n < 3 \cdot 2^{i-1}$ and $k \ge 3$ when $2^{i-1} < n < 2^i,$ then $\mathfrak{K}(n) \le 2^i+7$.
\end{lemma}

\begin{proof}
    Let $\varphi(d)$ be as in the proof of Lemma \ref{32x+d}. We wish to find some $c \le 2^{i-1} + 6$ such that $$\left\lfloor \frac{(a \cdot 2^j+d)c+\ell}{2^{i-1}}\right\rfloor \equiv_t \left\lfloor \frac{(a \cdot 2^j+d)c+\ell}{2^{i-1}}\right\rfloor + (a \cdot 2^j+d)$$
    for all $0\le \ell\le a\cdot 2^j+d$. It will then follow (by Proposition \ref{DivideBy2^i}) that the $(c+1)^\text{th}$ and $(c+1+2^{i-1})^\text{th}$ blocks of ${\bf t}$ are the same, which will prove the lemma. It suffices to find some $c\le 2^{i-1}+6$ such that the following three conditions hold:
    
\begin{enumerate}
    \item $$\left\lfloor\frac{(a \cdot 2^j+d)(c+1)}{2^{i-1}}\right\rfloor - \left\lfloor\frac{(a \cdot 2^j+d)c}{2^{i-1}}\right\rfloor \le 2.$$
    \item $$\left\lfloor\frac{(a \cdot 2^j+d)c}{2^{i-1}}\right\rfloor \equiv \varphi(d) \Mod {32}.$$
    \item There exists $p\in\{1, 2\}$ such that $$p \cdot 2^j < \left\lfloor\frac{(a \cdot 2^j+d)c}{2^{i-1}}\right\rfloor < (p+1) \cdot 2^j$$ and such that $\varphi(d)\equiv_t d+\varphi(d)$ if and only if $a+p\equiv_t p$. 
\end{enumerate}

    Indeed, suppose these three conditions hold. We may write $$\left\lfloor\frac{(a \cdot 2^j+d)c}{2^{i-1}}\right\rfloor = p\cdot 2^j + 32 v+\varphi(d),$$ where $0\le v < 2^{j-5}$. Then $$\left\lfloor\frac{(a \cdot 2^j+d)c}{2^{i-1}}\right\rfloor+(a\cdot 2^j+d) = (a+p)\cdot 2^j+32v+(d+\varphi(d)).$$  Note that $\varphi(d), d+\varphi(d) < 30$. Either $\varphi(d)+s \equiv_t d+\varphi(d)+s$ for all $0\leq s\leq 2$ or $\varphi(d)+s \not\equiv_t d+\varphi(d)+s$ for all $0\leq s\leq 2$. Thus, by condition (3), we have $t_p+t_v+t_{\varphi(d)+s} \equiv t_{a+p}+t_v+t_{d+\varphi(d)+s} \Mod 2$. Because $$t_{p\cdot 2^j+32v+\varphi(d)+s} \equiv t_p+t_v+t_{\varphi(d)+s} \Mod 2$$ and $$t_{(a+p)\cdot 2^j+32v+(d+\varphi(d)+s)} \equiv t_{a+p}+t_v+t_{d+\varphi(d)+s} \Mod 2$$ for all $0 \le s \le 2$, it follows that for all $0 \le s \le 2,$ $$p\cdot 2^j+32v+\varphi(d)+s \equiv_t (a+p)\cdot 2^j+32v+d+\varphi(d)+s.$$ Therefore, for all $\left\lfloor\frac{(a \cdot 2^j+d)(c)}{2^{i-1}}\right\rfloor \le x \le \left\lfloor\frac{(a \cdot 2^j+d)(c+1)}{2^{i-1}}\right\rfloor \le \left\lfloor\frac{(a \cdot 2^j+d)(c)}{2^{i-1}}\right\rfloor+2$, we have $x \equiv_t x+(a\cdot 2^j + d)$ by condition (1), so $\mathfrak{K}(n) \le 2^{i-1}+c+1 \le 2^i+7.$
    
    Now, we show such a $c$ exists. 
    Define $r$ as the positive integer less than $2^{k+4}$ such that $ra \equiv 1 \Mod {2^{k+4}}$ (which is well-defined as $a$ is odd). Then $$rn = r(a \cdot 2^j + d) = ra \cdot 2^j + rd = 2^{i+4} \cdot z + 2^j + rd$$ for some integer $z$. Since $rd > -2^{k+4} \cdot 16 = -2^{k+8}$ and $j \ge 3k+11 > k+8$, we have that $2^j + rd$ is positive. Now, consider the sequence $rn, 2rn, \dots$. Since $j \le i-2$, $d \le 13,$ and $0 \le r \le 2^{k+4},$ we have that $2^j + rd < 2^{i-1}$ for sufficiently large $i$, so $\varphi(d)2^{i-1}+(2^j+rd) < (\varphi(d)+1)2^{i-1} \le 2^{i+4}.$ 
    If we look at the remainder when we divide these terms by $2^{i+4}$, the remainder either increases by $2^j+rd$ or decreases by $2^{i+4}-(2^j+rd)$ when we change from $grn$ to $(g+1)rn$ for some $g$.
    
    
    Note that since $a$ is odd, there exists $p \in \{1, 2\}$ such that $a+p \equiv_t p$. There also exists $p \in \{1, 2\}$ such that $a+p \not\equiv_t p.$ Therefore, we can choose a value of $p$ accordingly based on $n$ to satisfy the second part of condition (3). 
    Define $g$ as the smallest integer such that $\left\lfloor\frac{grn}{2^{i-1}}\right\rfloor > p \cdot 2^j.$ Note that $g$ is positive as $\left\lfloor \frac{0 \cdot r n}{2^{i-1}}\right\rfloor = 0 < p \cdot 2^j.$ Then there exists $h \ge g$ such that $h-g \le ar$ and the remainder of $hr n$ when divided by $2^{i+4}$ is between $\varphi(d)2^{i-1}$ and $\varphi(d)2^{i-1} + (2^j+rd),$ inclusive. This is because $$(2^j+rd)(ar) \ge (2^j-13r)(ar) \ge (2^j-13 \cdot 2^{k+4})(2^{k+4}+1)$$ $$= 2^{i+4}+ 2^j - 13 \cdot 2^{2k+8} - 13 \cdot 2^{k+4} \ge 2^{i+4}$$ for sufficiently large $i$, as $j \ge 3k+11$. Therefore, the set of remainders when $grn, (g+1)rn, \dots, (g+ar)rn$ is divided by $2^{i+4}$ must intersect the set of integers between $\varphi(d) 2^{i-1}$ and $\varphi(d) 2^{i-1}+(2^j+rd)$, inclusive. Let $c_1 = hr.$ If the remainder when $c_1n$ is divided by $2^{i+4}$ is between $\varphi(d)2^{i-1}$ and $\varphi(d)2^{i-1} + (2^j-d)-1,$ inclusive, then the remainder when $(c_1 + 1) n = c_1 n + n$ is divided by $2^{i+4}$ is at most $$\varphi(d)2^{i-1}+(2^j-d)-1+(a \cdot 2^j + d) < \varphi(d)2^{i-1} + (a+1) \cdot 2^j \le (\varphi(d) + 3)(2^{i-1}).$$ This means that if we set $c=c_1,$ both conditions (1) and (2) are satisfied.
    
    Otherwise, $d > 0$ and the remainder when $c_1n$ is divided by $2^{i+4}$ is between $\varphi(d)2^{i-1} + (2^j-d)$ and $\varphi(d)2^{i-1} + (2^j+rd),$ inclusive. Now, define $c_2 = c_1+(a-1)r-1.$ If $c_1n \equiv \varphi(d)2^{i-1}+2^j+d' \Mod {2^{i+4}}$ for some $-d \le d' \le rd,$ then
    \begin{align*}
    c_2n &\equiv \varphi(d)2^{i-1}+2^j+d'+((a-1)r-1)(a\cdot 2^j+d) \Mod {2^{i+4}} \\
    &\equiv \varphi(d)2^{i-1}+2^j+d'+(a-1)2^j+(a-1)rd-a\cdot 2^j-d \Mod {2^{i+4}} \\
    &\equiv \varphi(d)2^{i-1}+d'+(a-1)rd-d \Mod {2^{i+4}},
    \end{align*}
    which is between $\varphi(d)2^{i-1}+(a-1)rd-2d$ and $\varphi(d)2^{i-1}+(a-1)rd+(rd-d),$ which equals $\varphi(d)2^{i-1}+ard-d.$ However, $a \ge 5$ since $k \ge 3$ if $n < 2^i$ and $k \ge 2$ if $n > 2^i.$ This means $(a-1)rd-2d > 0$ and therefore $\varphi(d)2^{i-1} < \varphi(d)2^{i-1}+(a-1)rd-2d.$ Also, $$ard-d < 3 \cdot 2^{k-1} \cdot 2^{k+4} \cdot 2^4 - d\le 2^{2k+9} -d \le 2^j-d,$$ which means that $c_2n$ when divided by $2^{i+4}$ has remainder between $\varphi(d)\cdot 2^{i-1}$ and $\varphi(d) \cdot 2^{i-1}+2^j-d-1,$ and $\left\lfloor \frac{c_2n}{2^{i-1}}\right\rfloor \equiv \varphi(d) \Mod {32}.$ However, note that $$(\varphi(d) \cdot 2^{i-1}+2^j-d-1)+(a\cdot 2^j+d) \le \varphi(d) \cdot 2^{i-1} + (a+1) \cdot 2^j-1$$ $$< \varphi(d) \cdot 2^{i-1}+3 \cdot 2^{i-1} = (\varphi(d)+3)2^{i-1}.$$  Therefore, $\left\lfloor \frac{(c_2+1)n}{2^{i-1}}\right\rfloor$ is congruent to one of $\varphi(d), \varphi(d)+1,$ or $\varphi(d)+2 \Mod {32}.$ But since $2 \cdot 2^{i-1} < n < 3 \cdot 2^{i-1},$ we must have that $\left\lfloor \frac{(c_2+1)n}{2^{i-1}}\right\rfloor-\left\lfloor \frac{(c_2)n}{2^{i-1}}\right\rfloor$ either equals $2$ or $3$, so $\left\lfloor\frac{(c_2+1)n}{2^{i-1}}\right\rfloor -\left\lfloor\frac{(c_2)n}{2^{i-1}}\right\rfloor = 2.$ Thus, both conditions (1) and (2) are satisfied if $c=c_2$.
            
    Whether or not $c = c_1$ or $c = c_2$ we have $$gr \le c \le gr+(ar)r+(a-1)r-1 \le (g-1)r+(ar+a)r.$$ Therefore, $\left\lfloor \frac{c n}{2^{i-1}}\right\rfloor \ge \left\lfloor \frac{(gr)n}{2^{i-1}}\right\rfloor > p \cdot 2^j$. Now, since $a \le 3 \cdot 2^{k-1}$ and $r \le 2^{k+4},$ 
    \begin{align*}
    \left\lfloor \frac{c n}{2^{i-1}}\right\rfloor &\le \left\lfloor \frac{(g-1)r n + (ar+a)rn}{2^{i-1}}\right\rfloor \\
    &\le p \cdot 2^j + 1 + \frac{(3 \cdot 2^{2k+3}+2^{k+4})(2^{k+4})(3 \cdot 2^{i-1})}{2^{i-1}} \\
    &\le p \cdot 2^j + 1 + 3 \cdot 2^{3k+9} < (p+1) \cdot 2^j,
    \end{align*}
    with the last inequality true since $3 \cdot 2^{3k+9}+1 < 2^{3k+11} \le 2^j.$ This proves that condition (3) holds for $c$.
    
    Finally, we show that condition (3) implies $c\le 2^{i-1}+6$. If $n > 2^i,$ since $(p+1) \cdot 2^j \le 3 \cdot 2^j \le a \cdot 2^j,$ we have that $\frac{(a \cdot 2^j-13)c}{2^{i-1}} < a \cdot 2^j$  and thus $c < \frac{a \cdot 2^j \cdot 2^{i-1}}{a \cdot 2^j-13}.$ As $a \cdot 2^j > 2^i+13,$ we have $\frac{a \cdot 2^j \cdot 2^{i-1}}{a \cdot 2^j-13} < 2^{i-1} \cdot \frac{2^{i}+13}{2^i} < 2^{i-1}+7,$ so $c \le 2^{i-1}+6.$ If $n < 2^i,$ recall that $c_1 \le (ar+g)r$ and $c_2 \le (ar+g)r + (a-1)r < (2ar+g)r.$ So we just have to check that $(2ar+g)r \le 2^i+6.$ However, since $\frac{grn}{2^{i-1}} \le (p+1) \cdot 2^j \le 3 \cdot 2^j$ and $n \ge 2^{i-1},$ we have $gr \le 3 \cdot 2^{j}.$ Also, $2ar^2 \le 2 \cdot 2^k \cdot (2^{k+4})^2 \le 2^{3k+9} \le 2^{j},$ which means that $(2ar+g)r \le 2^{j+2} \le 2^{i-1},$ since we are assuming that $i-3 \ge j.$
\end{proof}

\begin{lemma} \label{2^i+d}
    If $d \in \{1, 3, 5, 11, 13\}$ and $i$ is sufficiently large, then $\mathfrak{K}(2^i+d) \le 2^i + 5$.
\end{lemma}

\begin{proof}
    First, suppose that $d \in \{1, 5, 11, 13\}$. If $\left\lfloor\frac{c(2^i+d)}{2^i}\right\rfloor, \dots, \left\lfloor\frac{(c+1)(2^i+d)}{2^i}\right\rfloor$ are equivalent, respectively, to $\left\lfloor\frac{c(2^i+d)}{2^i}\right\rfloor + (2^i+d), \dots, \left\lfloor\frac{(c+1)(2^i+d)}{2^i}\right\rfloor + (2^i+d),$ then the $(c+1)^\text{th}$ and $(c+1+2^i)^\text{th}$ blocks of ${\bf t}$ are identical. Observe also that for $c \le 4$ and sufficiently large $i$, the sequences $\left\lfloor\frac{c(2^i+d)}{2^i}\right\rfloor, \dots, \left\lfloor\frac{(c+1)(2^i+d)}{2^i}\right\rfloor$ and $c, c+1$ are identical.
    
    Let $$\psi(d) = 
    \begin{cases} 
    2, & \text{if } d = 1; \\
    3, & \text{if } d = 5; \\
    0, & \text{if } d = 11; \\
    3, & \text{if } d = 13.
   \end{cases} $$
   
   For all $d \in \{1, 5, 11, 13\}$, it is straightforward to verify that $\psi(d) \equiv_t 2^i + d + \psi(d)$ and $\psi(d) + 1 \equiv_t 2^i + d + \psi(d) + 1$. Thus, by setting $c=\psi(d)$, we are done in the case $d\in\{1,5,11,13\}$.
   
   Next, assume $d = 3$. Note that if $\left\lfloor\frac{c(2^i+3)}{2^{i-1}}\right\rfloor + (2^i+3), \dots, \left\lfloor\frac{(c+1)(2^i+3)}{2^{i-1}}\right\rfloor + (2^i+3)$ are equivalent respectively to $\left\lfloor\frac{c(2^i+3)}{2^{i-1}}\right\rfloor + (2^{i+1}+6), \dots, \left\lfloor\frac{(c+1)(2^i+3)}{2^{i-1}}\right\rfloor + (2^{i+1}+6)$, then the $(c+1+2^{i-1})^\text{th}$ and $(c+1+2^i)^\text{th}$ blocks are the same. Choose $c = 4$. The sequences $\left\lfloor\frac{c(2^i+3)}{2^{i-1}}\right\rfloor + (2^i+3), \dots, \left\lfloor\frac{(c+1)(2^i+3)}{2^{i-1}}\right\rfloor + (2^i+3)$ and $2^i + 11, 2^i + 12, 2^i + 13$ are identical. Similarly, the sequences $\left\lfloor\frac{c(2^i+3)}{2^{i-1}}\right\rfloor + (2^{i+1}+6), \dots, \left\lfloor\frac{(c+1)(2^i+3)}{2^{i-1}}\right\rfloor + (2^{i+1}+6)$ and $2^{i+1}+14, 2^{i+1}+15, 2^{i+1}+16$ are identical. Since $11 \equiv_t 14$, $12 \equiv_t 15$, and $13 \equiv_t 16$, we have that $2^i+11 \equiv_t 2^{i+1}+14$, $2^i+12 \equiv_t 2^{i+1}+15$, and $2^i+13 \equiv_t 2^{i+1}+16$.
   
   For each $d$, we were able to choose an appropriate $c\leq 4$ such that $(c+1)^\text{th}$ and $(c+1+2^i)^\text{th}$ blocks of $\textbf{t}$ are identical. It follows that $\mathfrak{K}(2^i+d)\le 2^i+5$.
\end{proof}

We use Lemmas \ref{8x+1}, \ref{32x+d}, \ref{CloseToHighV2}, and \ref{2^i+d} to provide a strong upper bound on $\mathfrak{K}(n)$ for all odd $n$, which will help improve bounds for both $\gamma(k)$ and $\Gamma(k)-\gamma(k)$.

\begin{theorem} \label{limsupK}
    For odd $n$, we have $\limsup\limits_{n \to \infty} \frac{\mathfrak{K}(n)}{n} \le \frac{4}{3}.$
\end{theorem}

\begin{proof}
    If $2^i < n < 3 \cdot 2^{i-1} - 13$, then $\mathfrak{K}(n) \le 2^i \cdot (1 + o(1)) < n \cdot (1+o(1))$, where the inequality $\mathfrak{K}(n) \le 2^i \cdot (1 + o(1)) \le n \cdot (1+o(1))$ follows directly from Lemmas \ref{8x+1} through \ref{2^i+d}. If $3 \cdot 2^{i-1} - 13 \le n < 2^{i+1}$, then \cite[Lemma 14]{Defant} implies $\mathfrak{K}(n) \le 2^{i+1} \cdot (1 + o(1)) \le n \cdot \frac{4}{3} \cdot (1+o(1))$. We note, however, that there is an error in \cite[Lemma 14]{Defant}, which we fix in \ref{FixColin}.
\end{proof}

For $n = 2^i+1, 2^i+3,$ and $2^{2i}-3$, we determine exact values of $\mathfrak{K}(n)$. These values will be used to compute bounds for $\Gamma(k)-\gamma(k)$ for specific values of $k$, which will prove useful in understanding their asymptotic bounds.

\begin{lemma} \label{2^i+1 and 2^i+3}
    We have that $\mathfrak{K}(2^i+1) = 2^{i-1}+2$ and $\mathfrak{K}(2^i+3) = 2^i+5$ for all sufficiently large $i$.
\end{lemma} 

\begin{proof}
    To show $\mathfrak{K}(2^i+1) \le 2^{i-1}+2,$ note that if we set $c = 1,$ then $\left\lfloor \frac{c(2^i+1)}{2^{i-1}}\right\rfloor = 2$ and $\left\lfloor \frac{(c+1)(2^i+1)}{2^{i-1}}\right\rfloor = 4$. But for $i \ge 4$ and $x=2,3,4$, we have $x \equiv_t 2^{i-1}+x+1$. This means that $\mathfrak{K}(2^i+1) \le 2^{i-1}+c+1 = 2^{i-1}+2$ for $i \ge 4$ by Proposition \ref{DivideBy2^i}. 
    
    However, we know if the $j^\text{th}$ and $(j')^\text{th}$ blocks of length $2^i+1$ in ${\bf t}$ match, then $j \equiv j' \Mod {2^{i-1}}$. Hence, if $\mathfrak{K}(2^i+1) < 2^{i+1} + 2$, then the $1^\text{st}$ and $(2^i+1)^\text{th}$ blocks match. But the $(2^{i-1}+1)^\text{th}$ letter in the $(2^{i+1}+1)\text{th}$ block is $t_{2^{i-1}(2^i+1)+2^{i-1}} = 0$ while the $(2^{i-1}+1)^\text{th}$ letter in the first block is $t_{2^{i-1}} = 1$. Therefore, $\mathfrak{K}(2^i-1) = 2^{i-1}+2$.
    
    Note that $\mathfrak{K}(2^i+3) \ge 2^{i}-2$ as a consequence of \cite[Lemma 19]{Defant}. Suppose that $\mathfrak{K}(2^i+3) \le 2^i$. By Proposition \ref{ConvOfPrev}, there exists $d\in\{0,1,2\}$ such that the $(2^{i-1}-d)^\text{th}$ and $(2^{i}-d)^\text{th}$ blocks of ${\bf t}$ are the same. For these blocks to be the same, the sequences $\left\lfloor \frac{(2^i+3)(2^{i-1}-d-1)}{2^{i-1}}\right\rfloor, \dots, \left\lfloor \frac{(2^i+3)(2^{i-1}-d)-1}{2^{i-1}}\right\rfloor$ and $\left\lfloor \frac{(2^i+3)(2^{i}-d-1)}{2^{i-1}}\right\rfloor, \dots, \left\lfloor \frac{(2^i+3)(2^{i}-d)-1}{2^{i-1}}\right\rfloor$ must have the same respective binary digit sums modulo $2$. These sequences equal
    \begin{align*}
        2^i-4, 2^i-3, 2^i - 2, & \hspace{.2cm}\text{and} \hspace{0.2cm} 2^{i+1}-1, 2^{i+1},2^{i+1}+1, & \text{if}\hspace{0.2cm} d = 2; \\
        2^i-2, 2^i-1, 2^i, & \hspace{.2cm}\text{and} \hspace{0.2cm} 2^{i+1}+1, 2^{i+1}+2,2^{i+1}+3, & \text{if}\hspace{0.2cm} d = 1; \\
        2^i,2^i+1, 2^i+2, 2^i+3, & \hspace{.2cm}\text{and} \hspace{0.2cm} 2^{i+1}+2, 2^{i+1}+3,2^{i+1}+4,2^{i+1}+5, & \text{if}\hspace{0.2cm} d = 0.
    \end{align*} 
    Since $2^i-4$ has binary digit sum $i-2$ and $2^{i+1}-1$ has binary digit sum $i+1$, we have $d \neq 2$. We also know $d \neq 1$ since $2^i-2$ and $2^i-1$ have binary digit sums $i-1$ and $i$, respectively, while $2^{i+1}+1$ and $2^{i+1}+2$ both have binary digit sum $2$. Finally, $d \neq 0$ since $2^i$ has binary digit sum $1$ and $2^{i+1}+2$ has binary digit sum $2$.
    
    We may now assume $\mathfrak{K}(2^i+3) = 2^i + c$ for some $c\geq 1$. This is only possible if the $(2^i+c)^\text{th}$ block in ${\bf t}$ is the same as either the $c^\text{th}$ block or the $(2^{i-1}+c)^\text{th}$ block. If $c \ll i$ then the $c^\text{th}$ block of length $2^i+3$ of positive integers, when divided by $2^{i-1}$ term-by-term, gives the sequence $2c-2, 2c-1, 2c$, and the $(c+2^{i-1})^\text{th}$ and $(c+2^i)^\text{th}$ blocks of positive integers give $2^i+2c+1, 2^i+2c+2, 2^i+2c+3$ and $2^{i+1}+2c+4, 2^{i+1}+2c+5, 2^{i+1}+2c+6.$ Since $c \ll i$, adding a $2^{i+1}$ or $2^{i}$ term simply changes the letter $t_{2c+x}$ (from $1$ to $0$ or vice verse) for $1 \le x \le 6$. In other words, $c$ is the smallest positive integer such that at least two of $$t_{2c-2}t_{2c-1}t_{2c}, \overline{t_{2c+1}t_{2c+2}t_{2c+3}}, \overline{t_{2c+4}t_{2c+5}t_{2c+6}}$$ match. Just looking at $t_n$ for $0 \le n \le 16$ tells us that the first $c$ for which we have a match is $c = 5.$ Therefore,  $\mathfrak{K}(2^i+3) = 2^i + 5.$
\end{proof}

\begin{remark}
    Note that we already knew $\mathfrak{K}(2^i+3) \le 2^i+5$ for sufficiently large $i$ from Lemma \ref{2^i+d}; now we know equality holds.
\end{remark}

\begin{lemma} \label{2^{2i}-3}
    We have $\mathfrak{K}(2^{2i}-3) = 2^{2i}+10$.
\end{lemma}

\begin{proof}
    Let $n = 2^{2i}-3$. Because $3 \cdot 2^{2i-2} < 2^{2i}-3$ for $i \ge 2,$ if the $j^\text{th}$ and $(j')^\text{th}$ blocks of ${\bf t}$ match, then $j \equiv j' \Mod {2^{2i-1}}$. First, assume $\mathfrak{K}(n) \le 2^{2i-1}+2.$ Then either blocks $1$ and $2^{2i-1}+1$ or blocks $2$ and $2^{2i-1}+2$ match. However, this means either $x \equiv_t x+n$ for all $x$ such that $0 = \left\lfloor \frac{0}{2^{2i-1}}\right\rfloor \le x \le \left\lfloor \frac{n-1}{2^{2i-1}}\right\rfloor = 1$ or $x \equiv_t x+n$ for all $x$ such that $1 = \left\lfloor \frac{n}{2^{2i-1}} \right\rfloor \le x \le \frac{2n-1}{2^{2i-1}} = 3$. But $t_1 = 1$ and $t_{2^{2i-1}-2} = 0$, so $\mathfrak{K}(n) \ge 2^{2i-1}+3.$
    
    Next, suppose that $\mathfrak{K}(n) \le 2^{2i}$. Then there exists some $2 \le c \le 2^{2i-1}-1$ such that $x \equiv_t x+n$ for all $\left\lfloor \frac{cn}{2^{2i-1}} \right\rfloor \le x \le \left\lfloor \frac{(c+1)n-1}{2^{2i-1}} \right\rfloor$. However, any such $x$ must satisfy $3 \le x < 2^{2i-1}.$ Therefore, whenever $\left\lfloor \frac{cn}{2^{2i-1}} \right\rfloor \le x \le \left\lfloor \frac{(c+1)n-1}{2^{2i-1}} \right\rfloor$, we have $x \not\equiv_t x-3$. But since $\left\lfloor \frac{cn}{2^{2i-1}} \right\rfloor \le \left\lfloor \frac{(c+1)n-1}{2^{2i-1}} \right\rfloor-1$, there exists $x$ such that $t_x \neq t_{x-3}, t_{x+1} \neq t_{x-2}.$ Since either $x$ or $x+3$ is odd, we must have $t_x \neq t_{x+1}$ and $t_{x-3} \neq t_{x-2},$ so the only possible sequences for $t_{x-3}t_{x-2}t_{x-1}t_xt_{x+1}$ are $01010, 01110, 10101, 10001$. We know none of $01010, 10101, 111,$ and $000$ can appear as contiguous subwords of ${\bf t}$ because $\textbf{t}$ is overlap-free. This yields a contradiction, so $\mathfrak{K}(n) > 2^{2i}.$
    
    Now, assume $\mathfrak{K}(n) = 2^{2i}+c$ for $c < 2^{2i-1}$. This means that either $x \equiv_t x+2n$ for all $x$ such that $\left\lfloor \frac{(c-1)n}{2^{2i-1}}\right\rfloor \le x \le \left\lfloor \frac{cn-1}{2^{2i-1}}\right\rfloor$ or $x+n \equiv_t x+2n$ for all such $x$. Writing out the possible values of $x, x+n, x+2n$ for $1 \le c \le 10$ gives us the following table:
    
\begin{table}[ht]
\begin{tabular}{|c||l|l|l|}
    \hline
    $c$ & $x$ & $x+n$ & $x+2n$ \\ \hline
    $1$ & $0, 1$ & $2^{2i}-3, 2^{2i}-2$ & $2^{2i+1}-6, 2^{2i+1}-5$ \\[0.3ex] \hline
    $2$ & $1, 2, 3$ & $2^{2i}-2, 2^{2i}-1, 2^{2i}$ & $2^{2i+1}-5, 2^{2i+1}-4, 2^{2i+1}-3$\\ \hline
    $3$ & $3, 4, 5$ & $2^{2i}, 2^{2i}+1, 2^{2i}+2$ & $2^{2i+1}-3, 2^{2i+1}-2, 2^{2i+1}-1$\\ \hline
    $4$ & $5, 6, 7$ & $2^{2i}+2, 2^{2i}+3, 2^{2i}+4$ & $2^{2i+1}-1, 2^{2i+1}, 2^{2i+1}+1$ \\ \hline
    $5$ & $7, 8, 9$ & $2^{2i}+4, 2^{2i}+5, 2^{2i}+6$ & $2^{2i+1}+1, 2^{2i+1}+2, 2^{2i+1}+3$ \\ \hline
    $6$ & $9, 10, 11$ & $2^{2i}+6, 2^{2i}+7, 2^{2i}+8$ & $2^{2i+1}+3, 2^{2i+1}+4, 2^{2i+1}+5$ \\ \hline
    $7$ & $11, 12, 13$ & $2^{2i}+8, 2^{2i}+9, 2^{2i}+10$ & $2^{2i+1}+5, 2^{2i+1}+6, 2^{2i+1}+7$ \\ \hline
    $8$ & $13, 14, 15$ & $2^{2i}+10, 2^{2i}+11, 2^{2i}+12$ & $2^{2i+1}+7, 2^{2i+1}+8, 2^{2i+1}+9$ \\ \hline
    $9$ & $15, 16, 17$ & $2^{2i}+12, 2^{2i}+13, 2^{2i}+14$ & $2^{2i+1}+9, 2^{2i+1}+10, 2^{2i+1}+11$ \\ \hline
    $10$ & $17, 18, 19$ & $2^{2i}+14, 2^{2i}+15, 2^{2i}+16$ & $2^{2i+1}+11, 2^{2i+1}+12, 2^{2i+1}+13$ \\ \hline
\end{tabular}
\end{table}

    But it is straightforward to check the following for $i \ge 3$: 
\begin{align*}
    1 &\not\equiv_t 2^{2i+1}-5, \\
    2^{2i}-2 &\not\equiv_t 2^{2i+1}-5, \\
    5 &\not\equiv_t 2^{2i+1}-1, \\
    2^{2i}+2 &\not\equiv_t 2^{2i+1}-1, \\
    9 &\not\equiv_t 2^{2i+1}+3, \\
    2^{2i}+5 &\not\equiv_t 2^{2i+1}+2, \\
    2^{2i}+8 &\not\equiv_t 2^{2i+1}+5, \\
    13 &\not\equiv_t 2^{2i+1}+7, \\
    2^{2i}+10 &\not\equiv_t 2^{2i+1}+7, \\
    15 &\not\equiv_t 2^{2i+1}+9, \\
    2^{2i}+13 &\not\equiv_t 2^{2i+1}+10.
\end{align*}

    These show that $\mathfrak{K}(n) > 2^{2i}+9.$ However, 
\begin{align*}
    2^{2i}+14 &\equiv_t 2^{2i+1}+11, \\
    2^{2i}+15 &\equiv_t 2^{2i+1}+12, \\
    2^{2i}+16 &\equiv_t 2^{2i+1}+13,
\end{align*}
    which means that indeed $\mathfrak{K}(n) = 2^{2i}+10.$
\end{proof}

\begin{lemma} \label{17over6}
    For all sufficiently large $i$, there exists an odd $n$ such that $\frac{17}{6} \cdot 2^i - 96 < n < \frac{17}{6} \cdot 2^i$ and $\mathfrak{K}(n) \le 2^i+6.$
\end{lemma}

\begin{proof}
    Note that if $i$ is sufficiently large, then $\frac{17}{6} \cdot 2^i-96 > \frac{14}{5} \cdot 2^i$. This means that if $c = 6$, then $\left\lfloor \frac{(c-1)n}{2^i}\right\rfloor, \dots, \left\lfloor \frac{cn-1}{2^i} \right\rfloor$ is the sequence $14, 15, 16$. Note that $t_{15} = t_{17} \neq t_{16}$ and if we choose $n \equiv 1 \Mod {32},$ then $14+n \equiv 15 \Mod {32}$ so $t_{14+n} = t_{16+n} \neq t_{15+n}$. Thus, it suffices to find an $n$ such that $t_{14} = t_{14+n}$ and $n \equiv 1 \Mod {32}$. There exist $3$ values of $n$, say $n_1,n_2,n_3$, between $\frac{17}{6} \cdot 2^i - 96$ and $\frac{17}{6} \cdot 2^i$ satisfying the latter condition, such that $n_3 = n_2+32 = n_1+64.$ It suffices to show $14+n_1 \equiv_t 14+n_2 \equiv_t 14+n_3$ cannot hold, because then $t_{14} = t_{14+n_I}$ for some $n_i.$ If $14+n_1 \equiv_t 14+n_2 \equiv_t 14+n_3$, then $\left\lfloor \frac{14+n_1}{32}\right\rfloor \equiv_t \left\lfloor \frac{14+n_2}{32}\right\rfloor \equiv_t \left\lfloor \frac{14+n_3}{32}\right\rfloor$ since $n_1\equiv n_2\equiv n_3 \Mod {32}$. This means that $3$ consecutive integers are Thue-Morse equivalent, a contradiction. Thus, the $6^\text{th}$ block and $(2^i+6)^\text{th}$ block of length $n$ are the same for some $n \in \{n_1, n_2, n_3\}$.
\end{proof}


    

\section{Bounds on $\gamma(k)$ and on $\Gamma(k)-\gamma(k)$} \label{BoundsGamma}

First, we directly use the results of the previous section to improve bounds on $\gamma(k)$.

\begin{theorem} \label{limsupgamma}
    We have $\limsup\limits_{k \to \infty} \frac{\gamma(k)}{k} = \frac{3}{2}$.
\end{theorem}

\begin{proof}
    Since $\limsup_{k \to \infty} \frac{\gamma(k)}{k} \le \frac{3}{2}$ is true by Theorem \ref{DefantTheorem}, proven in \cite{Defant}, we only need to prove that $\limsup_{k \to \infty} \frac{\gamma(k)}{k} \ge \frac{3}{2}$. Suppose $n=a\cdot 2^j+d$ is between $2^i$ and $3 \cdot 2^{i-1}-13$. If $j \ge \frac{3i+11}{4},$ then it follows from Lemmas \ref{CloseToHighV2} and \ref{2^i+d} that $\mathfrak{K}(n) \le 2^i+7$. Otherwise, $\mathfrak K(n)\leq 2^i + 3 \cdot 2^j +28 \le 2^i + 3 \cdot 2^{(3i+11)/4} + 28$ by Lemmas \ref{8x+1} and \ref{32x+d}. For $n < 2^i,$ we have that $\mathfrak{K}(n) < 2^i(1+o(1))$ \cite[Lemma 17]{Defant}, which means that we can choose $k = 2^i(1+o(1))$ such that $\gamma(k) \ge \frac{3}{2} \cdot 2^i - 13.$ Taking this limit as $i$ goes to infinity, we get the desired result.
\end{proof}

\begin{theorem} \label{liminfgamma}
    We have $\liminf\limits_{k \to \infty} \frac{\gamma(k)}{k} \ge \frac{3}{4}.$
\end{theorem}

\begin{proof}
    This follows from Theorem \ref{limsupK}, since $\gamma(k)$ is clearly the smallest $n$ such that $\mathfrak{K}(n) \ge k.$ Note that for all $\varepsilon$ we can choose $N$ such that $\gamma(n) \le (\frac{4}{3} + \varepsilon)n$ for all $n \ge N.$ Now, let $K = \max_{n \le N} \mathfrak{K}(n).$ Then for all $k \ge K,$ $\gamma(k) \ge N.$ If $\gamma(k) = n,$ then $\mathfrak{K}(n) \ge k$. But $\mathfrak{K}(n) \le (\frac{4}{3} + \varepsilon)n$, so $(\frac{4}{3} + \varepsilon)n \ge k.$ Therefore, $\gamma(k) \ge k \cdot \frac{1}{4/3 + \varepsilon}$ for all sufficiently large $k$, so the conclusion follows.
\end{proof}

Thus, we have proven the second part of Defant's conjecture and improved the lower bound for $\liminf\limits_{k \to \infty} \frac{\gamma(k)}{k}$ to $\frac{3}{4}.$

It turns out that $\Gamma(k)-\gamma(k)$ also grows linearly in $k.$ Note that $\limsup_{k \to \infty} \frac{\Gamma(k)-\gamma(k)}{k}$ is clearly positive and bounded above by $\frac{9}{4}$, since $$\frac{9}{4} \ge \limsup\limits_{k \to \infty} \frac{\Gamma(k)}{k}-\liminf\limits_{k \to \infty} \frac{\gamma(k)}{k} \ge \limsup\limits_{k \to \infty} \frac{\Gamma(k)-\gamma(k)}{k},$$ since $\limsup_{k \to \infty} \frac{\Gamma(k)}{k} = 3$ by Theorem \ref{DefantTheorem} (d) and $\liminf_{k \to \infty} \frac{\gamma(k)}{k} \ge \frac{3}{4}$ by Theorem \ref{liminfgamma}. However, if we look at $\liminf_{k \to \infty} \frac{\Gamma(k)-\gamma(k)}{k}$ similarly using Theorems \ref{DefantTheorem} (c) and \ref{limsupgamma}, we get $$\liminf\limits_{k \to \infty} \frac{\Gamma(k)-\gamma(k)}{k} \ge \liminf\limits_{k \to \infty} \frac{\Gamma(k)}{k} - \limsup\limits_{k \to \infty} \frac{\gamma(k)}{k} = \frac{3}{2} - \frac{3}{2} = 0,$$ so it is not obvious from our bounds on $\Gamma(k)$ and $\gamma(k)$ that $\Gamma(k)-\gamma(k)$ grows linearly in $k$. However, we can use the lemmas of the previous section and results from \cite{Defant} to find stronger bounds on $\Gamma(k)-\gamma(k)$.

\begin{theorem}
    We have $$\frac{1}{2} \le \liminf\limits_{k \to \infty} \frac{\Gamma(k)-\gamma(k)}{k} \le \frac{3}{4}.$$
\end{theorem}

\begin{proof}
    Suppose that $2^{i} \le k < 2^{i+1}$ and $i$ is sufficiently large. If $k \le 2^i+5$, then $\mathfrak{K}(2^i+3) \ge k$, so $\gamma(k) \le 2^i+3.$ However, if $i$ is even, then \cite[Lemma 10]{Defant} proves that $\Gamma(2^{i-1} + 3 \cdot 2^{i/2 + 2} + 49) \ge 3 \cdot 2^{i-1}-2^{i/2-2}+1$. But both $\gamma$ and $\Gamma$ are clearly nondecreasing functions, so $\Gamma(k) \ge 3 \cdot 2^{i-1}-2^{i/2-2}+1$ for sufficiently large $i$. Therefore, $\Gamma(k)-\gamma(k) \ge 2^{i-1}-2^{i/2-2}-2 \ge (1/2-o(1))\cdot k.$ Similarly, if $i$ is odd, \cite[Lemma 10]{Defant} proves that $\Gamma(2^{i-1}+2^{(i-1)/2}+2) \ge 3\cdot 2^{i-1}-2^{(i-1)/2}+1$. Thus, $\Gamma(k) \ge 3\cdot 2^{i-1}-2^{(i-1)/2}+1$, and $\Gamma(k)-\gamma(k) \ge 2^{i-1}-2^{(i-1)/2}-2 \ge (1/2-o(1)) \cdot k.$
    
    Alternatively, if $k > 2^i+5,$ then \cite[Lemma 10]{Defant} clearly implies that there exist sequences $a_n, b_n$ converging to $0$ such that $\mathfrak{K}((3-a_n)\cdot 2^i) < 2^i \cdot (1+b_n).$ Also, $$\mathfrak{K}(3\cdot 2^{i-1}+1) > \frac{5\cdot 2^{2i-1}}{3\cdot 2^{i-1}+1} \ge 2^i \cdot (1+b_n)$$ for sufficiently large $i,$ with the first inequality true by \cite[Lemma 19]{Defant}. Therefore, $\gamma(k) \le 3\cdot 2^{i-1}+1$ if $k \le 2^i(1+b_n).$ However, $\mathfrak{K}(2^{i+1}+1) = 2^i+2 < k$, which means that $\Gamma(k) \ge 2^{i+1}+1$. Therefore, $\Gamma(k)-\gamma(k) \ge 2^{i-1} \ge (1/2-o(1)) \cdot k.$
    
    Finally, if $k>2^i(1+b_n),$ then $\Gamma(k) \ge (3-a_n) 2^i$. As $\mathfrak{K}(2^{i+1}+3) > 2^{i+1} > k,$ we have that $\gamma(k) \le 2^{i+1}+3$, which means $\Gamma(k)-\gamma(k) \ge (1-a_n) \cdot 2^i - 3 \ge (1/2-o(1)) k$ since $k < 2^{i+1}.$ This proves $\liminf_{k \to \infty} \frac{\Gamma(k)-\gamma(k)}{k} > \frac{1}{2}$ and proves that $\Gamma-\gamma$ indeed grows linearly in $k$, as we know $\Gamma(k)-\gamma(k) \le \Gamma(k)$ and thus $\limsup_{k \to \infty} \frac{\Gamma(k)-\gamma(k)}{k} < \infty.$
    
    The upper bound is quite direct, since $$\liminf\limits_{k \to \infty} \frac{\Gamma(k)-\gamma(k)}{k} \le \liminf\limits_{k \to \infty} \frac{\Gamma(k)}{k}-\liminf\limits_{k \to \infty} \frac{\gamma(k)}{k} \le \frac{3}{4}. \qedhere$$
\end{proof}

Unfortunately, we were unable to establish stronger bounds on this value. It appears unlikely that the true value of $\liminf_{k \to \infty} \frac{\Gamma(k)-\gamma(k)}{k}$ is $\frac{3}{4}$.

\begin{theorem}
    We have $$\frac{11}{6} \le \limsup_{k \to \infty} \frac{\Gamma(k)-\gamma(k)}{k} \le \frac{9}{4}.$$
\end{theorem}

\begin{proof}
    The second inequality has been proven already. The first inequality comes from setting $k = 2^{2i}+9.$ Note that $\gamma(2^{2i}+9) \le 2^{2i}-3$ for sufficiently large $i$ by Lemma \ref{2^{2i}-3}, but $\Gamma(2^{2i}+9) \ge \frac{17}{6} \cdot 2^{2i}-96$ for sufficiently large $i$ by Lemma \ref{17over6}. Therefore, $$\frac{\Gamma(k)-\gamma(k)}{k} \ge \frac{\frac{11}{6} \cdot 2^{2i}-93}{2^{2i}+9} \ge \frac{11}{6}(1-o(1)),$$ so the conclusion follows by letting $i$ go to infinity.
\end{proof}

\section{Conclusion and Further Directions} \label{Conclusion}

A natural way to view these results is to consider an infinite grid $2\BN-1 \times \BN$, where $(x, y)$ is shaded if $y < \mathfrak{K}(x).$ In the interval $[3\cdot 2^{i-1}-13, 3\cdot 2^i-13) \times \BN,$ we know that every element $(x, y)$ where $x$ is odd and $y \le 2^i$ is shaded, and every element $(x, y)$ where $x$ is odd and $y \ge 2^{i+1} \cdot (1+o(1))$ is unshaded. 

Unfortunately we know very little about the in-between region, with the exception of some values of $\mathfrak{K}(n)$ for a few special values of $n$. This makes improving bounds for $\liminf_{k \to \infty} \frac{\gamma(n)}{n}$ and $\liminf_{k \to \infty} \frac{\gamma(n)}{n} \frac{\Gamma(n)-\gamma(n)}{n}$ difficult. While we know $\mathfrak{K}(n)$ is between $2^i$ and $2^{i+1}(1+o(1))$ for $3 \cdot 2^{i-1} < n < 3\cdot 2^i,$ and while naturally it appears that $\mathfrak{K}(n)$ should usually be close to $2^i,$ there are certain values of $n$ such as $3 \cdot 2^{i-1}+1$ and $2^{i+1}+3$ where $\mathfrak{K}(n)$ can get very large. Improving the asymptotic bounds on $\frac{\gamma(k)}{k}$ and $\frac{\Gamma(k)-\gamma(k)}{k}$ rely on understanding for what $n$ is $\mathfrak{K}(n)$ much larger than expected to be.

We make the following conjecture about the growth of $\mathfrak{K}(n)$, which if true would prove useful in understanding $\Gamma(k)-\gamma(k)$ as well:

\begin{conjecture} \label{MyConjecture}
    There exists some sequence $a_i$ converging to $0$ such that $\mathfrak{K}(n) \le \mathfrak{K}(3 \cdot 2^i+1)$ for all odd $n$ with $3 \cdot 2^i < n < 2^{i+2} \cdot (1-a_i).$
\end{conjecture}

\begin{figure}
    \centering
    \includegraphics[width=4.8in]{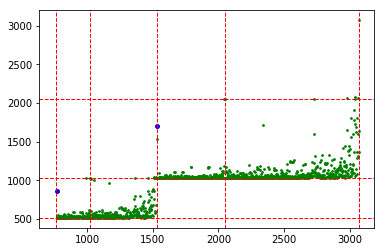}
    \caption{A graph of $\mathfrak{K}(n)$ for odd $n$ between $768 = 3 \cdot 2^8$ and $3072 = 3 \cdot 2^{10}.$ The points for $n = 3 \cdot 2^8+1$ and $n = 3 \cdot 2^9+1$ are in blue. Dashed vertical lines are drawn at $x = 768, 1024, 1536, 2048, 3072,$ and dashed horizontal lines are drawn at $y = 512, 1024, 2048.$ Note that $\mathfrak{K}(3 \cdot 2^i + 1)$ appears to be larger than $\mathfrak{K}(n)$ for all $n \le 2^{i+2}(1-\varepsilon)$ for some small $\varepsilon.$}
    \label{Visualization}
\end{figure}

Conjecture \ref{MyConjecture} is supported by Figure \ref{Visualization}. Note that Conjecture \ref{MyConjecture} implies that $\limsup_{k \to \infty} \frac{\gamma(n)}{n} = \frac{9}{10}$ and $\liminf_{k \to \infty} \frac{\Gamma(k)-\gamma(k)}{k} = \frac{1}{2}.$ Moreover, it implies that $\limsup_{k \to \infty} \frac{\Gamma(k)-\gamma(k)}{k} \le 2$ since for all $2^i \le k < 2^{i+1},$ $\Gamma(k) < 3 \cdot 2^i$ and $\gamma(k) > 2^i \cdot (1-o(1)).$

Finally, define the infinite word $\mathbf{t}^n = t_0^nt_1^nt_2^n \cdots$ so that $t_i^n$ is the sum of the digits of $i$ in base $n$, reduced mod $n$. Similar questions to the ones we studied could be asked about corresponding values of $\mathfrak{K}, \gamma, \Gamma$ for $\mathbf{t}^n.$ 

\vspace{1cm}
\section*{Acknowledgments}
This research was funded by NSF grant 1358659 and NSA grant H98230-16-1-0026
as part of the 2016 Duluth Research Experience for Undergraduates (REU).

The author would like to thank Prof.\ Joe Gallian for running a wonderful REU and supervising the research, as well as Colin Defant, a fellow student at the REU whose research he encouraged me to look at. The author would also like to thank Gallian and Defant as well as Levent Alpoge and the anonymous reviewers for helpful comments on editing the paper.

\vspace{1cm}

\vspace{1cm}

\appendix

\section{Fix to Error in Lemma 14 in Defant's paper} \label{FixColin}

In \cite[Lemma 14]{Defant}, Defant attempts to prove $\mathfrak{K}(n) < \left(1 + \frac{37}{n}\right) \cdot 2^{\lceil \log_2 n\rceil}$ for all $n \equiv 29 \Mod {32},$ though his proof has a flaw. In his proof, he claims that
$$\bigcup\limits_{r = 9}^{17} \left(\frac{17}{2r}, \frac{10}{r+1}\right) = \left(\frac{1}{2}, 1\right),$$
which is incorrect. Here, we fix this issue by proving something slightly weaker which is sufficient for our purposes. We prove that for all sufficiently large $n$ such that $n \equiv 29 \Mod {32},$ $\mathfrak{K}(n) < \left(1 + o(1)\right) \cdot 2^{\lceil \log_2 n\rceil}$. One can verify this is sufficient to fix any issues with \cite[Lemma 14]{Defant} as well as any future issues in Defant's to get the required bounds on $\gamma(n)$ and $\Gamma(n)-\gamma(n)$.

The following lemma corresponds closely to Lemma \ref{32x+d}. In addition to Lemma \ref{CloseToHighV2} for $n \equiv 29 \Mod {32}$ where $2^{i-1} < n < 2^i,$ it will take care of almost all cases where $n \equiv 29 \Mod {32}$ and $2^{i-1} < n < 2^i.$

\begin{lemma} \label{Fix1}
    Suppose that $i$ is sufficiently large and $2^{i-1} \le n \le 2^{i}.$ Also, suppose that $n = a \cdot 2^j - 3$ where $a$ is odd and $j \ge 5$. Then $\mathfrak{K}(n) \le 2^i + 4 \cdot 2^j + 39$.
\end{lemma}

\begin{proof}
    Recall that $\varphi(-3) = 18,$ where $\varphi$ was the function from Lemma \ref{32x+d}. Now, for some nonnegative integer $y$, let $x = 32y + 18.$ Since $n \le 2^i,$ there must exist some integer $c$ such that $\left\lfloor \frac{cn}{2^i} \right\rfloor = x.$ We also must have $\left\lfloor \frac{(c+1)n-1}{2^i} \right\rfloor \le x+1.$
    
    Now, suppose that we chose $y$ such that $y \equiv_t y + a \cdot 2^{j-5}.$ Then $x \equiv_t x+n$ and $x+1 \equiv_t x+n+1$ since $x = 32y+18$ and $x+n = 32(y+a \cdot 2^{j-5}) + 15.$ Therefore, the $(c+1)$th and $(c+1+2^i)$th blocks of size $n$ are equal, which means that since $c \le 2(x+1)$ as $n > 2^{i-1}$ and $\lfloor \frac{cn}{2^i} \rfloor = x,$ $\mathfrak{K}(n) \le 2^i + 2x+3 = 2^i + 64y + 39.$
    
    Now, the only goal is to find a sufficiently small $y$. However, since $a$ is odd then we must have either $1 \equiv_t a+1$ or $2 \equiv_t a+2.$ This is because $1 \equiv_t 2$ but $a+1$ is even so $a+1 \not\equiv_t a+2$. If $s \in \{0, 1, 2\}$ is the smallest nonnegative integer such that $s \equiv_t a+s,$ then $s \cdot 2^{j-5} = y$ satisfies $y \equiv_t y + a \cdot 2^{j-5}.$ Therefore, $\mathfrak{K}(n) \le 2^i + 64 \cdot (2 \cdot 2^{j-5}) + 39 = 2^i + 4 \cdot 2^j + 39.$
\end{proof}

Now, the only cases remaining are $n = 2^i-3$ and $n = 3 \cdot 2^{i-2}-3.$

\begin{lemma} \label{Fix2}
    $\mathfrak{K}(2^i-3) \le 2^i+10$ for all sufficiently large $i$.
\end{lemma}

\begin{proof}
    Note that $2^i+14 \equiv_t 2^{i+1}+11, 2^{i}+15 \equiv_t 2^{i+1} + 12,$ and $2^i+16 \equiv_t 2^{i+1}+13.$ This means that $x \equiv_t x+(2^i-3)$ for all $2^i+14 \le x \le 2^i+16,$ i.e., for all $\lfloor \frac{(2^{i-1}+9)(2^i-3)}{2^{i-1}} \rfloor \le x \le \left\lfloor \frac{(2^{i-1}+10)(2^i-3)-1}{2^{i-1}} \right\rfloor$. Therefore, setting $c = 2^i+9$ and using Proposition \ref{DivideBy2^i} gives us the desired result.
\end{proof}

\begin{lemma} \label{Fix3}
    $\mathfrak{K}(3 \cdot 2^{i-1} - 3) \le 2^i + 14$ for all sufficiently large $i$.
\end{lemma}

\begin{proof}
    Note that $19 \equiv_t 3 \cdot 2^{i-1}+16$ and $20 \equiv_t 3 \cdot 2^{i-1} + 17.$ Therefore, we have that $x \equiv_t x+(3 \cdot 2^{i-1}-3)$ for all $19 \le x \le 20,$ so $x \equiv_t x+(3 \cdot 2^{i-1}-3)$ for all $\left\lfloor \frac{13 \cdot (3 \cdot 2^{i-1} - 3)}{2^i} \right\rfloor \le x \le \left\lfloor \frac{14 \cdot (3 \cdot 2^{i-1}-3)-1}{2^i} \right\rfloor.$ Now, using Proposition \ref{DivideBy2^i} with $c = 13$ gives us the desired result.
\end{proof}

Lemmas \ref{CloseToHighV2}, \ref{Fix1}, \ref{Fix2}, and \ref{Fix3}, in addition to the work of \cite{Defant}, give us that for all odd $2^{i-1} \le n \le 2^i,$ $\mathfrak{K}(n) \le 2^i \cdot (1+o(1))$ for sufficiently large $i$. This fixes all issues in Defant's paper.

\end{document}